\documentclass[10pt]{amsart}
\usepackage[english]{babel}
\usepackage{fancyhdr}
\usepackage{stmaryrd}
\usepackage{calrsfs}
\usepackage{amsmath}
\usepackage[nospace,noadjust]{cite}
\usepackage{amsfonts}
\usepackage{amssymb,enumerate,enumitem}
\usepackage{amsthm}
\usepackage{cite}
\usepackage{comment}
\usepackage{color}
\usepackage[all]{xy}
\usepackage{hyperref}
\usepackage{lineno}
\usepackage{mathtools}
\usepackage{tikz-cd}
\usepackage{xy}
\input xy
\xyoption{all}
\usepackage{stmaryrd}
\usepackage{calrsfs}
\usepackage{caption}

\def\C{\mathbb{C}}
\def\Q{\mathbb{Q}}
\def\R{\mathbb{R}}
\def\N{\mathbb{N}}

\def\Pr{{\rm Prin}}

\DeclareMathOperator{\Supp}{Supp}

\DeclareMathOperator{\lc}{lc}
\DeclareMathOperator{\ord}{ord}

\providecommand\ldb{\llbracket}
\providecommand\rdb{\rrbracket}
\newtheorem{Th}{Theorem}[section]
\newtheorem{Lm}[Th]{Lemma}
\newtheorem{Pro}[Th]{Proposition}
\newtheorem{Co}[Th]{Corollary}
\newtheorem{Ep}[Th]{Example}
\newtheorem{Df}[Th]{Definition}
\newtheorem{rk}[Th]{Remark}

\usepackage{graphicx} 

\begin{document}

\setcounter{page}{1}

\title{On Integral Domains with Prime Divisor Finite Property}

\author{Mohamed Benelmekki}

\address{Department of Mathematics, Facult\'e des Sciences et Techniques, Beni Mellal University, P.O. Box 523, Beni Mellal, Morocco}
\email{Mohamed Benelmekki:  med.benelmekki@gmail.com}
\date{\today}
\begin{abstract} An integral domain $D$ is called a \emph{prime-divisor-finite domain} (PDF-domain) if every nonzero element has only finitely many nonassociate prime divisors. A domain $D$ is said to be a \emph{tightly prime-divisor-finite domain} (TPDF-domain) if it is a PDF-domain and every nonzero nonunit element admits at least one prime divisor. In this paper, we study TPDF-domains. We investigate some basic properties of these domains and examine the behavior of the TPDF property under standard constructions such as localization, $D+M$ constructions, and polynomial rings.
\end{abstract}
\maketitle
\section{Introduction}

Factorization theory studies how elements of an integral domain decompose into irreducible elements and how far this behavior is from unique factorization. Even though unique factorization domains (UFDs) provide the strongest form of control, many domains fail to be UFDs but still satisfy weaker finiteness or divisibility conditions. Understanding these intermediate classes is a central problem in modern commutative algebra.

A natural way to measure the complexity of factorization is through finiteness conditions on divisors. An integral domain $D$ is called an \emph{IDF-domain} (irreducible-divisor-finite domain) if every nonzero nonunit element of $D$ has only finitely many irreducible divisors up to associates \cite{GW75}. Similarly, $D$ is called a \emph{PDF-domain} (prime-divisor-finite domain) if every nonzero nonunit element has only finitely many prime divisors up to associates \cite{MO}. These conditions prevent any element from having infinitely many essentially different irreducible or prime divisors, thus imposing strong constraints on the divisibility structure of the domain.\medskip

While finiteness controls the number of divisors, it does not guarantee their existence. For this reason, existence conditions also play an important role. Following P.L. Clark \cite{Cl17}, an integral domain $D$ is called a \emph{Furstenberg domain} if every nonzero nonunit element admits an irreducible divisor. Strengthening this notion, we say that $D$ is a \emph{strong Furstenberg domain} if every nonzero nonunit element admits a prime divisor.\medskip

Combining existence and finiteness conditions leads naturally to the notion of a \emph{tightly prime divisor finite domain} (TPDF-domain). A TPDF-domain is a strong Furstenberg domain that is also a PDF-domain. Equivalently, every nonzero nonunit element admits a prime divisor and has only finitely many prime divisors up to associates. The class of TPDF-domains is an important subclass of AP-domains (cf. Proposition~\ref{P1}), namely domains in which every irreducible element is prime. It naturally generalizes the class of UFDs, maintaining strong structural control over factorizations while allowing non-unique factorization. This property was introduced in \cite{BDKN}, where A.~Bodin et al. use the term ``near UFD'' for what we call a TPDF-domain, in the context of a version of Hilbert's Irreducibility Theorem over rings. In particular, it is shown there that near UFDs satisfy the coprime Schinzel Hypothesis in a version of Hilbert's Irreducibility Theorem over integral domains. \medskip

The aim of this paper is to study TPDF-domains and AP-domains. We investigate their fundamental properties and examine the behavior of the TPDF property under standard constructions  such as localization, $D+M$ constructions, and polynomial rings. The remainder of the paper is organized as follows. In Section~\ref{sec2}, we recall the necessary background and establish the notation used throughout the paper. We review several definitions and results concerning factorization properties in integral domains,  
and we present monoid domains as a central tool in our study. In Section~\ref{sec3}, we introduce TPDF-domains and study their basic properties. Section~\ref{sec4} investigates how the TPDF property behaves under the $D+M$ construction and localization.  
\section{Background}\label{sec2}
\subsection{General Notation}
We let $\mathbb{N}$, $\mathbb{Z}$, $\mathbb{Q}$, $\mathbb{R}$, and $\mathbb{C}$ denote the positive integers, the integers, the rational numbers, the real numbers, and the complex numbers, respectively. The sets $\mathbb{Z}_+$ 
 and  $\mathbb{Q}_+$ will denote the monoid of non-negative, integers and rational numbers, respectively. Throughout this paper, all monoids are assumed to be commutative, cancellative, torsion-free semigroups with identity.
\subsection{Factorization in integral domains}
			
Let $M$ be a monoid, written multiplicatively. An element $b \in M$ is \emph{unit} if there exists $c \in M$ such that $b c = 1$. We let $U(M)$ denote the group of unit elements of $M$, and we say that $M$ is reduced if $U(M)$ is trivial. When $M$ is the multiplicative monoid of an integral domain $D$, the unit elements of $M$ are precisely the units of $D$. \smallskip
			
	Let $D$ be an integral domain.	Let $a,b\in D$, we say that $b$ is divided by $a$ or $a$ divides $b$ if there exists $c\in D$ such that $b=ac$; in this case we can write $a|b$. We say that $a$ and $b$ are associate if  $a| b$ and $b|a$; we write $a\sim b$. A nonunit element $a\in D$ is called atom (or irreducible element) if $a=bc$ for some $b,c\in D$  implies that $b\in U(D)$ or $c\in U(D)$. A nonunit element $p \in D$ is called \emph{prime} if 
$p \mid ab$ implies $p \mid a$ or $p \mid b$ for all $a,b \in D$. We will note the set of atoms (resp., prime elements) of $M$ by $\mathcal{A}(D)$ (resp., $\mathcal{P}(D)$). \smallskip
			
			Following Cohn \cite{C68}, we say that a domain  $D$ is \emph{atomic} if every nonunit $a\in D$ can be expressed as a finite product of atoms in $D$. An atomic domain $D$ is called a \emph{bounded factorization domain} (BFD) if, for each nonzero nonunit element $a \in D$, there exists $n_0 \in \mathbb{N}$ such that whenever $a = a_1 \cdots a_n$ is a factorization of $a$ into atoms, 
one has $n \le n_0$ \cite{AAZ90}. We say that $D$ is an \emph{IDF-domain} (resp., a \emph{TIDF-domain}) if every nonzero nonunit element of $D$ has only a finite (resp., a  nonempty finite) set of nonassociate irreducible divisors \cite{GW75, GZ22}. The domain $D$  is said to be a \textit{prime-divisor-finite} domain (PDF-domain), or a $\text{GD}(1)$ in the terminology of \cite{MO}, if every nonzero element has only a finite number of nonassociate prime divisors. The integral domain $D$ is an \emph{AP-domain} provided that every irreducible of $D$ is prime. An integral domain with no atoms, called an \emph{antimatter domain} \cite{CDM99}, is vacuously an AP-domain. 

Following P.L. Clark \cite{Cl17}, we say that an integral domain $D$ is a Furstenberg domain if every nonzero nonunit element of $D$ has at least an irreducible divisor in $D$. Inspired by the above definition, we introduce the following stronger version and then define 
a TPDF-domain.

\begin{Df}$ $
\begin{itemize}
    \item An integral domain $D$ is called a \emph{strong Furstenberg domain} if every nonzero nonunit element of $D$ admits a prime divisor.
    \smallskip

    \item A strong Furstenberg PDF-domain is called a \emph{tightly PDF-domain} (\emph{TPDF-domain}).
\end{itemize}
\end{Df}
 
Note that, in \cite{BDKN} the authors use the term  ``near UFD" for ``TPDF-domain".\medskip

\subsection{Content of Polynomials and PSP-Domains}
 
  Let $D$ be an integral domain with quotient field $K$, and let $\mathcal{F}(D)$ denote the set of nonzero fractional ideals of $D$. For  $A \in \mathcal{F}(D)$, let $A^{-1}$ denote the fractional ideal $$(D:_K A)=\{x\in K\mid xA\subseteq D\}.$$  Let $f \in K[T_1,\ldots,T_r]$, the content of $f$,  denoted by $A_f$, is the fractional ideal of $D$ generated by the coefficients of $f$. For  $0\not=f \in D[T_1,\ldots,T_r]$,  the polynomial  $f$ is called primitive (over $D$)  if the coefficients of $f$ have no nontrivial divisor in $D$.   The polynomial  $f$ is said to be super-primitive if $A_f^{-1} = D$. Let $I$ be a finitely generated ideal of $D$. Following \cite{AS75},  $I$ is called \emph{primitive} if it is contained in no proper principal ideal of $D$, and $I$ will be called \emph{super-primitive} if $I^{-1} = D$. The domain $D$ is called a \emph{PSP-domain} if every primitive polynomial $f\in D[T]$ is super-primitive. Note that, $D$ is a PSP-domain if and only if every primitive ideal of $D$ is super-primitive, see \cite[Proposition~1.2]{AS75}. It is well known that every GCD domain is a PSP-domain and that, in turn, 
every PSP-domain is an AP-domain \cite[Proposition~3.2]{AZ07}.

  \subsection{Monoid domains}
Let $D$ be an integral domain, and let $S$ be a commutative cancellative torsion-free monoid, written additively. Denote by $\prec$ a total order on $S$. The monoid domain of $S$ over $D$ is defined by $D[S]=\lbrace \sum_i a_iX^{s_i}|a_i\in D \mbox{ and } s_i\in S\rbrace$. Note that $D[S]$ is an integral domain and each nonzero element $f\in D[S]$ has a unique representation in the form $$ f= a_{1}X^{s_1}+a_{2}X^{s_2}+\cdots+a_{n}X^{s_n},$$ where $n\in \mathbb{Z}_+$, $0\neq a_{i}\in D$ and ${s_i}\in S$ $(i=1,\ldots,n)$ such that $s_1\prec s_2 \prec\cdots\prec s_n$.
 The subset $\Supp(f)=\lbrace s_1,s_2,\ldots,s_n \rbrace$ of $S$ is called the support of $f$. The element $s_0:=\ord f$ is called the order of $f$. The constant term of $f$ will be denoted by $f(0)$.   The element $s_n$ (resp., $a_n$) is called the degree (resp., the leading coefficient) of $f$ and denoted by $\deg f$ (resp., $\lc(f)$).   Notice that  the units of $D[S]$ are  the monomials  $aX^s$, where  $a\in D$ and $s\in S$ are both  units  \cite[Corollary~4.2]{GP74}. An excellent reference for monoid domains is \cite{G84}.

 \section{Strong Furstenberg PDF-domains}\label{sec3}

Recall that an integral domain $D$ is called a \textit{strong Furstenberg} domain if every nonzero nonunit element of $D$ has at least a prime divisor in $D$. In terms of $\Pr(D)$, the set of nonzero principal integral ideals of $D$,  $D$ is a strong Furstenberg domain if and only if every  nonzero proper principal ideal of $D$ is contained in at least a prime principal ideal. A strong Furstenberg domain is obviously a Furstenberg domain. However, a Furstenberg domain need not be a strong Furstenberg domain. For example, consider an atomic domain which is not UFD (e.g. $K[X^2, X^3]$ where $K$ is a field), and hence an atomic domain also need not be a strong Furstenberg domain. 

\begin{Pro}\label{P1}
An integral domain $D$ is a strong Furstenberg domain if and only if it is a Furstenberg AP-domain. 
\end{Pro}
\begin{proof}
Assume that $D$ is a strong Furstenberg domain. Clearly, $D$ is a Furstenberg.  On the other hand,  let $a$ be an irreducible element of $D$. Since $D$ is a strong Furstenberg domain, $a$ has a prime divisor $p$ in $D$. But $a$ is irreducible in $D$, then $a$ and $p$ are associates in $D$, and hence $a$ is prime in $D$.  Therefore, $D$ is an AP-domain. The reverse implication is trivial.	
\end{proof}
 It is clear that a UFD is a strong Furstenberg domain. At the other extreme, a strong Furstenberg domain need not be a UFD. For instance, fix a prime number and let $D = \mathbb{Z}_{(p)} + T\mathbb{C}[[T]]$. Then, $p$ is the only prime element of $D$, and each nonunit element of $D$ is divisible by $p$. Thus $D$ is a strong Furstenberg domain. However, $D$ is not atomic, and hence it is not a UFD, see \cite[Examples 4.4 and 6.3]{GZ22}. In general we have the following relation between these conditions. 
 
 \begin{Co}\label{P2}
An integral domain $D$ is a UFD  if and only if it is an atomic strong Furstenberg domain. 
\end{Co}
\begin{proof}
The result is immediate since a strong Furstenberg domain is an AP-domain.	
\end{proof}
The next proposition investigates the ascent of the strong Furstenberg property from an integral domain to its polynomial ring.
\begin{Pro} \label{RA45}
Let $D$ be an integral domain with quotient field $K$. Then the following conditions are equivalent.
\begin{enumerate}
\item $D[T]$ is a strong Furstenberg domain.

 \smallskip
     \item $D$ is a strong Furstenberg domain, and every non-constant primitive polynomial over $D$ has a prime factorization in $D[T]$. 
   \smallskip
  \item $D$ is a strong Furstenberg, $D$ is a PSP-domain, and every non-constant irreducible polynomial in $D[T]$ is irreducible in $K[T]$.
\end{enumerate}
\end{Pro}
\begin{proof}   (1)$\Rightarrow $(2) Suppose that $D$ is a strong Furstenberg domain. Clearly $D$ is a strong Furstenberg domain. Let $f\in D[T]\setminus D$ be a primitive polynomial over $D$. Since $f$ is primitive, no nonunit of $D$ can occur in a factorization of $f$ in $D[T]$. Thus, by a degree argument, $f$ has an atomic factorization in $D[T]$, say $f=\prod_{i=1}^{n} f_i$, where each $f_i$ is irreducible in $D[T]$. By Proposition~\ref{P1}, for every $i=1,\ldots,n$, the polynomial $f_i$ is prime in $D[T]$. Hence $f$ has a prime factorization in $D[T]$.\smallskip

 (2)$\Rightarrow $(3) Suppose that $D$ is a strong Furstenberg domain and that every non-constant primitive polynomial over $D$ has a prime factorization in $D[T]$.  Let us prove that $D$ is a PSP-domain.   Thus, let $f$ be a primitive polynomial over $D$. We may assume that $f$ is not constant. Then $f$ has a prime factorization in $D[T]$, say $f=\prod_{i=1}^{n} f_i$, where each $f_i$ is prime in $D[T]$. Thus, it follows from \cite[Theorem~A]{Tang} that each $f_i$ is super-primitive over $D$. Hence, \cite[Theorem~F]{Tang} implies that $f$ is super-primitive, since it is a finite product of super-primitive polynomials. Thus, $D$ is a PSP-domain. Now let $f$ be a non-constant irreducible polynomial in $D[T]$. Then $f$ is primitive over $D$, and hence it has a prime factorization in $D[T]$. Since $f$ is irreducible, this factorization must be trivial, and therefore $f$ itself is prime in $D[T]$. Consequently, by \cite[Theorem~A]{Tang}, $f$ is irreducible in $K[T]$.\smallskip
 
 (3)$\Rightarrow $(1) Let $f$ be a nonzero nonunit element of $D[T]$. Since $D$ is a strong Furstenberg domain, we may suppose that $f$ is not constant in  $D[T]$. If $f$ is not primitive over $D$, then there exists a nonunit $a\in D$ such that  $f=af'$ for some $f'\in D[T]$. Since $D$ is a strong Furstenberg domain, $a$ has a prime divisor $b\in D$. But, it is easy to see that $b$ is also a prime divisor of $f$, and consequently of $f$, in $D[T]$. Hence, in this case, $f$ has at least one prime divisor in $D[T]$.  Otherwise, $f$ is primitive. By a degree argument, $f$ has an atomic factorization in $D[T]$, say $f=\prod_{i=1}^{n} f_i$, where each $f_i\in D[T]\setminus D$ irreducible. In particular, each $f_i$  is primitive, and hence super-primitive since $D$ is a PSP-domain. On the other hand, each $f_i$ is irreducible in $K[T]$ by hypothesis. Then  it follows from \cite[Theorem~A]{Tang} that each $f_i$ is prime in $D[T]$. In particular, $f$ has a prime divisor in $D[T]$. Therefore, $D[T]$ is a strong Furstenberg domain. 
\end{proof}

\begin{rk}
An alternative proof of the equivalence $(1) \Leftrightarrow (3)$ in the previous proposition follows from Proposition~\ref{P1} together with \cite[Proposition]{MR78} and \cite[Proposition~4.7]{GZ22}.
\end{rk}
Recall that an integral domain $D$  is said to be a PDF-domain (or a $\text{GD}(1)$) if every nonzero element has only a finite number of nonassociate prime divisors. Note that $D$ is a PDF-domain if and only if every nonzero principal ideal of $D$ is contained in only finitely many principal prime ideals of $D$. Clearly a UFD is a IDF-domain, and an IDF-domain is a PDF-monoid. Following \cite{MO}, for a constant-free ideal $I$ of $D[T]$ (i.e., $I \cap D = \{0\}$), we say that $D$ is \emph{$I$-robust} if only finitely many prime elements of $D$ become units in $D[T]/I$. We say $D$ is \emph{robust} if it is $I$-robust for every constant-free ideal $I$. An integral domain $D$ is robust if and only if it is a $\text{GD}(1)$ (i.e., a PDF-domain) domain see [4, Proposition 1.3]. Then,  it follows from \cite[Proposition 1.5]{MO} that a BFD is also a PDF-domain. On the other hand, $D$ is robust if and only if $D[T]$ is robust \cite[Proposition 1.4]{MO}. Thus, $D$ is a PDF-domain if and only if $D[T]$ is a PDF-domain. Clearly, this statement also holds for any family $\{T_\alpha\}_{\alpha \in \Lambda }$ of indeterminates.  Therefore, we include the following lemma for later citation.

\begin{Lm} \label{Lem1}
	Let $D$ be an integral domain. The following conditions are equivalent.
	\begin{enumerate}
		\item $D$ is a PDF-domain.
		\smallskip
		\item $D[T]$ is a PDF-domain.
		\smallskip
		
		\item $D[\{T_\alpha\}_{\alpha \in \Lambda }]$ is a PDF-domain for every nonempty set $\{T_\alpha\}_{\alpha \in \Lambda }$ of indeterminates. 
	\end{enumerate}
\end{Lm}

Recall that an integral domain $D$ is a TPDF-domain if $D$ is a strong Furstenberg PDF-domain; equivalently,  every nonzero nonunit of $D$ has a nonempty finite set of nonassociate prime divisors in $D$.  In terms of $\Pr(D)$, $D$ is a TPDF-domain if and only if every nonzero proper principal ideal of $D$ is contained in only a nonempty finite set of principal prime ideals of $D$.  

\begin{Co} \label{pr1}
	The following conditions are equivalent for an integral domain $D$.
	\begin{enumerate}
		\item $D$ is a TPDF-domain.
		\smallskip
		
		\item $D$ is a strong Furstenberg IDF-domain.
		\smallskip
		
		\item $D$ is an AP TIDF-domain.	
	\end{enumerate}
\end{Co}
\begin{proof}
$(1)\Leftrightarrow (2)$ Follows from the definition of a TPDF-domain.

$(2)\Leftrightarrow (3)$ This is an immediate consequence of Proposition \ref{P1}.
\end{proof}

Following \cite{Sh71}, we say that an integral domain $D$ is \emph{Archimedean} if $\cap_{n\in \N} \;a^n D=(0)$ for every nonunit $a\in D$. 

\begin{Co} \label{pr}
	The following conditions are equivalent for an integral domain $D$.
	\begin{enumerate}
	\item $D$ is an atomic TPDF-domain.
		\smallskip
	\item $D$ is a UFD.
		\smallskip
		
	\item $D$ is an Archimedean TPDF-domain.
		\smallskip
	
	\item $D$ is a TPDF-domain and  $\bigcap_{n \in \N} p^n D = (0)$ for every prime element $p\in D$.		
	\end{enumerate}
\end{Co}
\begin{proof}
$(1)\Leftrightarrow (2)$ Follows from Proposition \ref{P1} and Corollary \ref{P2}, and the implications $(2)\Rightarrow (3)\Rightarrow (4)$ follow immediately. 

$(4)\Rightarrow (2)$ Since $D$ is a TPDF-domain,  Corollary \ref{pr1} implies that $D$ is an AP TIDF-domain. By the assumption, the condition (d) of  \cite[Corollary 4.6]{GZ22} holds in the AP-domain $D$, and therefore $D$ is a UFD. 
\end{proof}
The following corollary investigates the ascent of the TPDF property to polynomial rings. It follows directly from Proposition~\ref{RA45} together with Lemma~\ref{Lem1}.
\begin{Co}
 Let $D$ be an integral domain with quotient field $K$.  The following conditions are equivalent.
  \begin{enumerate}
  \item $D[T]$ is a TPDF-domain. 
   \smallskip
     \item $D$ is a TPDF-domain, and every non-constant primitive polynomial over $D$ has a prime factorization in $D[T]$. 
   \smallskip
  \item $D$ is a TPDF-domain, $D$ is a PSP-domain, and every non-constant irreducible polynomial in $D[T]$ is irreducible in $K[T]$.
  \end{enumerate}
\end{Co}
\section{The $D+M$ construction and localisation}\label{sec4}

This section is devoted to investigating the behavior of the AP and TPDF properties under the $D+M$ construction and localization. Let $T$ be an integral domain that can be written in the from $T = K + M$, where  $K$ is a subfield of $T$ and $M$ is a nonzero maximal ideal of $T$. Let $D$ be a subring of $K$ and $R = D + M$. We begin this section by considering the AP-domains through the lens of the $D+M$ construction. 

\begin{Lm}\label{dr} Let $T$ be an integral domain of the form $K+M$, where  $K$ is a subfield of $T$ and $M$ is a nonzero maximal ideal of $T$.  Let $D$ subring of $K$ and $R = D + M$. Then the following statements hold.
\begin{enumerate}
\item $\mathcal{A}(R)\cap D=\mathcal{A}(D)$.
\smallskip
\item $\mathcal{P}(R)\cap D=\mathcal{P}(D)$.
\smallskip
\item For any ideal $P\subseteq T$ and any $x\in R$  such that $P\cap R=xR$, $P=xT$.
\end{enumerate} 
\end{Lm}
\begin{proof}
(1) Let $d\in\mathcal{A}(R)\cap D$ such that $d=d_1d_2$ for some  $d_1,d_2\in D$. Since $d$ is irreducible in $R$, either $d_1\in U(R)$ or $d_2\in U(R)$. Since $U(D)= D\cap U(R)$,  it follows that $d_1\in U(D)$ or $d_2\in U(D)$, and so $d\in \mathcal{A}(D)$. Hence $\mathcal{A}(R)\cap D\subseteq\mathcal{A}(D)$.  Conversely, suppose that $d\in\mathcal{A}(D)$ such that $d=(d_1+m_1)(d_2+m_2)$ for some $d_i\in D$  and $m_i\in M$, $i=1,2$.  Note that $d_1\neq 0$ and $d_2\neq 0$ since otherwise $d\in M$. Then,  it follows from the equality $d=(d_1+m_1)(d_2+m_2)$ that $d=d_1d_2$ and $1=(1+d_1^{-1}m_1)(1+d_2^{-1}m_2)$. In particular, the elements $(1+d_1^{-1}m_1)$ and $(1+d_2^{-1}m_2)$ are units of $R$. Since $d$ is irreducible in $D$, either $d_1\in U(D)$ or $d_2\in U(D)$.   Say, for example, that $d_1\in U(D)$. Then  $(d_1+m_1)= d_1(1+d_1^{-1}m_1)\in U(R)$, and hence $d \in \mathcal{A}(R)\cap D$. Thus $\mathcal{A}(D)\subseteq\mathcal{A}(R)\cap D$. Therefore, $\mathcal{A}(D)=\mathcal{A}(R)\cap D$.\smallskip

(2) Let $p\in\mathcal{P}(R)\cap D$ such that $d_0p=d_1d_2$ for some $d_0,d_1,d_2\in D$. Since $p$ is  prime in $R$, $p$ divides $d_1$ or $d_2$ in $R$. Say, for example, that $d_1=p(d+m)$  for some $d\in D$  and $m\in M$. Then $d_1=pd$, and hence $p$ divides $d_1$ in $D$. Therefore, $p$ is prime in $D$. Conversely, assume that $p$ is  prime in $D$ such that $r_0p=r_1r_2$ where $r_i=d_i+m_i$, $d_i\in D$  and $m_i\in M$, $i=0,1,2$. If $d_i=0$ for either $i=1$ or $i=2$,  then $r_i=p(p^{-1}m_i)$ with $p^{-1} m_i \in  M\subseteq R$, then we are done. Otherwise, $d_0p=d_1d_2$, and $p$ divides $d_1$ or $d_2$ in $D$ by the fact that $p$ is  prime in $D$. Say, for example, that $d_1=pd$ for some $d\in D$. Thus $r_1=pr_1'$, where $r_1':=(d+p^{-1}m_1)\in R$. Therefore $p$ is prime in $R$.
\smallskip

(3)  Let $P$ be an ideal of $T$ and $x\in R$  such that $P\cap R=xR$.  We have $x \in P \cap R$, which implies $x \in P$. Since $P$ is an ideal of $T$, it follows that $xT \subseteq P$. Conversely, let $y=k+m\in P$, where $k\in K$ and $m\in M$.  If $k=0$, then $y\in M\subseteq R$. Hence $y\in P \cap R $, and so $y\in xR$. Since $xR \subseteq xT$, $y\in xT$. Otherwise, $y=k(1+m')$, where $m'=k^{-1}m$.   It is clear that $1+m'\in P\cap R$, and then $1+m'\in xR$. As $xR \subseteq xT$, $1+m'\in xT$. Since $xT$ is an ideal of $T$, $k(1+m')=y\in xT$. Therefore $P \subseteq xT$. 
\end{proof}
We first consider the case when $D$ is not a field.
\begin{Th} \label{Th1}
	Let $T$ be an integral domain of the form $K+M$, where  $K$ is a subfield of $T$ and $M$ is a nonzero maximal ideal of $T$.  Let $D$ subring of $K$ and $R = D + M$. Assume that $D$ is not a field. Then the following statements hold.
\begin{enumerate}
\item  If $D$ and $T$ are AP-domains, then $R$ is an AP-domain. \smallskip

\item If $R$ is an AP-domain, then  $D$ is an AP-domain.
\smallskip

\item If $R$ is an AP-domain and each element of $\mathcal{A}(T)\cap M$ is prime in $T$, then   $T$ is an AP-domain.
\end{enumerate} 
\end{Th}
\begin{proof} 
(1)  Assume that $D$ and $T$ are AP-domains. Let $x$ be an irreducible element of $R$. Note that no element of $M$ is irreducible in $R$. Indeed,  for every nonzero nonunit $d \in D$ and every $m \in M$ we have $m = d(d^{-1}m)$ with $d^{-1} m \in  M$. Thus, there exist $d\in D^*$ and $m \in M$ such that $x=d(1+m)$. Since $x$ is  irreducible in $R$, either $d\in U(R)$ or  $1+ m\in U(R)$. If $d\in U(R)$, then $1+m$ is irreducible in $R$, and hence it is irreducible in $T$ by \cite[Lemma~1.5(2)]{CMZ}. Since $T$ is an AP-domain, $1+ m$ is prime in $T$. Then it follows from \cite[Lemma~1.5(3)]{CMZ} that $1+ m$, and hence $x$, is prime in $R$. Otherwise,  $d$ is irreducible $R$, then it is   irreducible in $D$ by Lemma \ref{dr}(1). Since  $D$ is an AP-domain,   $d$ is prime in $D$, and hence it is prime in $R$ by Lemma \ref{dr}(2). Therefore $x$ is prime in $R$. Consequently, $R$ is an AP-domain. 
\smallskip

 (2) Assume that $R$ is an AP-domain. Let $d$ be an irreducible element of $D$. Then, by Lemma~\ref{dr}(1), $d$ is irreducible in $R$.  Since $R$ is an AP-domain,  $d$ is prime in  $R$. Then it follows from Lemma~\ref{dr}(2) that $d$ is prime in  $D$. Therefore, $D$ is an AP-domain.
 \smallskip
 
 (3)  Assume that $R$ is an AP-domain and that each element of $\mathcal{A}(T)\cap M$ is prime in $T$. Let $x=k+m$ be an irreducible element of $T$, where $k\in K$ and $m\in M$.  If $k=0$, then $x\in M\cap \mathcal{A}(T)$. Hence $x$ is a prime element of $T$ by hypotheses.  Otherwise,  $k\neq 0$, and  then it follows that $x=k(1+m')$, where $m':=k^{-1}m\in M$. Since $k\in U(T)$, $1+m'$ is irreducible in  $T$.  Then, by  \cite[Lemma~1.5(2)]{CMZ},   $1+m'$ is irreducible  in  $R$.  Since $R$ is an AP-domain,  $1+m'$ is prime in  $R$. Then  $(1+m')R$ is a prime ideal of $R$. By \cite[Theorem~1.3]{CMZ}, either $(1+m')R=R\cap P$, where $P$ is a prime ideal of $T$, or $(1+m')R=P_0+ M$, where $P_0$ is a prime ideal of $D$. If the latter case holds, a contradiction is generated. Indeed, the equality  $(1+m')R=P_0+ M$ implies $1+m'\in P_0+M$. Then $1\in P_0$, and hence $P_0=D$  which is impossible since  $P_0$ is a prime ideal of $D$.  Therefore, $(1+m')R=R\cap P$ for some prime ideal $P$ of $T$. Then Lemma \ref{dr}(3) implies that  $P=(1+m')T$,  and hence $(1+m')$ is a prime element of $T$. Therefore,  $x$ is  prime in $T$. Consequently, $T$ is an AP-domain. 
\end{proof}
\begin{Co}
 Let $D$ be an integral domain with quotient field $K$, and let $L$ be a field extension of $K$. Let $R = D + X L[X]$ and $R' = D + X L \ldb X \rdb$. Assume that $D$ is not a field. Then $R$ (resp.,  $R'$) is an AP-domain  if and only if  $D$ is an  AP-domain.
\end{Co}
\begin{proof} The direct implication follows from Theorem~\ref{Th1}(2). The converse follows immediately from Theorem~\ref{Th1}(1), since $T=L[X]$ (resp., $T'=L\ldb X\rdb$) is a UFD, and hence an AP-domain.
\end{proof}
\begin{rk} If $D$ is a field, Theorem \ref{Th1}(1) may fails. To see this, let $T=\C[X]$ and $R=\R+X\C[X]$. Note that $T$ is a UFD, and hence an AP-domain.  On the other hand, $D$ is vacuously an AP-domain. However, since $-(iX)^2=X^2$, it is clear that $iX \in R$ is  irreducible but not prime in $R$. Then $R$ is not an AP-domain.
\end{rk}

\begin{Th} \label{Th}
	Let $T$ be an integral domain of the form $K+M$, where  $K$ is a subfield of $T$ and $M$ is a nonzero maximal ideal of $T$.  Let $D$ be a subfield  of $K$ and $R = D + M$.  Then $R$ is an AP-domain if and only if  $T$ is an AP-domain and each element of $\mathcal{A}(R)\cap M$ is prime in $R$. 
\end{Th}
\begin{proof} 
 Assume that $R$ is an AP-domain. Obviously, each element of $\mathcal{A}(R)\cap M$ is prime in $R$. Let $x=k+m$ be an irreducible element of $T$, where $k\in K$ and $m\in M^*$. We have two cases: 
 \smallskip

 \textit{Case 1:} $k=0$. Then $x=m\in R$. It is easy to see that $x$ is irreducible in  $R$.  Since $R$ is an AP-domain,  $x$ is prime in  $R$, and hence $xR$ is a prime ideal of $R$. By \cite[Theorem~1.3]{CMZ}, either $xR=R\cap P$, where $P$ is a prime ideal of $T$, or $xR=P_0+ M$, where $P_0$ is a prime ideal of $D$. In the first case, it follows from Lemma \ref{dr}(3) that $P=xT$, and hence $x$ is a prime element of $T$. In the latter case, since $D$ is a field, $P_0=(0)$, and so $M=xR$ is a prime ideal of $R$. Since $x\in M$, $M\subseteq xR\subseteq xT\subseteq M$. Thus,  $xT=M$. Since $M$ is an ideal of $T$, it follows that $xT$ is a prime ideal of $T$, and therefore $x$ is a prime element of $T$.
  \smallskip
 
 \textit{Case 2:} $k\neq 0$. Then $x=k(1+k^{-1}m)$. It is clear that $m':=k^{-1}m\in M$. Since $k\in U(T)$, $1+m'$ is irreducible in  $T$. Then, by  \cite[Lemma~1.5(1)]{CMZ},    $1+m'$ irreducible  in  $R$.  Since $R$ is an AP-domain,  $1+m'$ is prime in  $R$. Then  $(1+m')R$ is a prime ideal of $R$. By \cite[Theorem~1.3]{CMZ}, either $(1+m')R=R\cap P$, where $P$ is a prime ideal of $T$, or $(1+m')R=P_0+ M$, where $P_0$ is a prime ideal of $D$. The first case implies, by  Lemma \ref{dr}(3),  that $P=(1+m')T$,  and then $(1+m')$ (and hence $x$) is a prime element of $T$. The latter case is impossible; indeed, since $D$ is a field, $P_0=(0)$, and so $M=(1+m')R$ is a prime ideal of $R$. In particular,  $1=(1+m')-m'\in M$, whence $M=R$ which is a contradiction. \\ In conclusion,  $x$ is a prime element of $T$. Consequently, $T$ is an AP-domain.
 \smallskip
 
 Conversely, assume that $T$ is an AP-domain and that  elements of $\mathcal{A}(R)\cap M$ are prime in $R$. Let $x$ be an irreducible element of $R$. If $x=m\in M$, then $x$ is prime in $R$ by hypotheses. Otherwise, as previously seen,  up to multiplication by a $d\in D^*\subseteq U(R)$,  $x$ is of the form $1+ m$ for some $m \in M^*$. Moreover, $1+ m$ is irreducible in $T$ since it is irreducible in $R$, see \cite[Lemma~1.5(2)]{CMZ}. Thus, since $T$ is an AP-domain, $1+ m$ is  prime in $T$. Then it follows from \cite[Lemma~1.5(3)]{CMZ} that $1+ m$, and hence $x$, is a prime element of $R$. 
 \smallskip
\end{proof}
\begin{Ep}
Let $T$ denotes the monoid algebra $\C [\Q_+]$ of the Puiseux monoid $\Q_+$ over $\C$.  Let $M = \left\{ f \in T : f(0) = 0 \right\}$ the ideal of $T$ consisting of all elements $f$ with a zero constant term (i.e., $\ord f>0$). One can easily check that $M$ is a maximal ideal of $T$ such that  $T = \C + M$. Consider the subring $R = \Q + M$ of $T$. For every $f\in M$, $X^s$ divides $f$ in $R$, where $s=\frac{\ord f}{2}\in \Q_+$.  Thus no element of  $M$ is irreducible in $R$, and then $\mathcal{A}(R)\cap M=\emptyset$. On the other hand, $T$ is an antimatter domain (cf. \cite[Corollary~3]{BEGraz}),  hence an AP-domain. Then it follows from Theorem \ref{Th} that $R$ is an AP-domain. Note that $R$ is not an antimatter domain. This can be verified by observing that $R$ contains the irreducible element $f(X) = X-2$. \end{Ep}

Next, we consider the (strong) Furstenberg, PDF, and TPDF properties under the $D+M$ construction. Recall that a TPDF-domain is an AP TIDF-domain (see Corollary \ref{pr1}).

\begin{Th}\label{tN}
Let $T$ be an integral domain of the form $K+M$, where  $K$ is a subfield of $T$ and $M$ is a nonzero maximal ideal of $T$.  Let $D$ subring of $K$ and $R = D + M$.  Assume that $D$ is not a field. Then the following statements hold.
\begin{enumerate}
\item  $R$ is a Furstenberg domain if and only if $D$ is a  Furstenberg domain and every nonunit of the form $1+m$,  where $m\in M$, has at least an irreducible divisor in $R$. 
		\smallskip 
		
		\item  If $R$ is a (strong) Furstenberg domain, then $D$ is a (strong) Furstenberg domain. 
		\smallskip 
		
		\item  If $D$ and $T$ are (strong) Furstenberg domains, then $R$ is a (strong) Furstenberg domain. 
		\smallskip 
		
\item  If $R$ is a PDF-domain, then $\mathcal{P}(D)$ is finite up to associates.
\smallskip 

\item  If $T$ is a PDF-domain and  $\mathcal{P}(D)$ is finite up to associates, then $R$ is a PDF-domain.
\smallskip 

\item  If $R$ is a TPDF-domain, then $D$ is TPDF-domain with  $\mathcal{P}(D)$ is finite up to associates.
\smallskip 
 
\item  If $T$ and $D$ are TPDF-domains with $\mathcal{P}(D)$ is (nonempty) finite up to associate,  then $R$ is a TPDF-domain. 
\end{enumerate}
\end{Th}
\begin{proof}
(1)  Suppose that $R$ is a Furstenberg domain. Let $d$ be a nonzero nonunit element of $D$. Since $d$ is a nonunit of $R$, it has an irreducible divisor $r_1$ in $R$. Then there exists $r_2\in R$ such that $d=r_1r_2$. Say, $r_i=d_i+m_i$, where $d_i\in D$  and $m_i\in M$, $i=1,2$. Note $d_i\neq 0$ since otherwise $d\in M$. Thus $d=d_1d_2$, and hence $(1+d_1^{-1}m_1), (1+d_2^{-1}m_2)\in U(R)$. Since $r_1$ is irreducible in $R$ and $r_1=d_1(1+d_1^{-1}m_1)$, the fact that $(1+d_1^{-1}m_1)\in U(R)$ implies that  $d_1\in\mathcal{A}(R)\cap D$. Then it follows from  Lemma~\ref{dr}(1) that $d_1$ is an irreducible divisor of $d$ in $D$. Therefore, $D$ is a Furstenberg domain. The second condition is obvious. 
\smallskip

 Conversely, suppose that $D$ is a  Furstenberg domains and that every nonunit of the form $1+m$, where $m\in M$, has at least an irreducible divisor in $R$. Let $x=d+m$ be a nonzero nonunit element of $R$, where $d\in D$ and $m\in M$. If $d\neq 0$, then $x=d(1+m')$, where $m'=d^{-1}m\in M$. By hypotheses, the element $(1+m')$, and hence  $x$, has at least an irreducible divisor in $R$. Otherwise, let $l$ a nonzero nonunit element of $D$, and then $x=l(l^{-1}m)$ with $l^{-1}m\in M$. Since  $D$ is a  Furstenberg domain,  $l$ has a divisor $l_1\in\mathcal{A}(D)$ in $D$.  It follows from Lemma~\ref{dr} that $l_1\in\mathcal{A}(R)\cap D$, and hence $l_1$ is an irreducible divisor of $x$ in $R$.
 \smallskip

(2) Follows from the part (1), Proposition~\ref{P1}, and Theorem~\ref{Th1}(2).
\smallskip

(3) Assume that $D$ and $T$ are Furstenberg domains. By the part (1), it suffices to show that every nonunit of the form $1+m$, where $m\in M$, has at least an irreducible divisor in $R$. Let $m\in M$ such that $1+m$ is a nonunit of $R$. Since $T$ is a Furstenberg domain, $1+m$ has has an irreducible divisor $t_1$ in $T$. Note that $t_1$ (also every divisor of $1+m$ in $T$) is of the form  $u_1(1+m_1)$, where $u_1\in K^*$ and $m_1\in M$. Thus, there exists $t_2=u_2(1+m_2)\in R$ such that $1+m=u_1(1+m_1)u_2(1+m_2)$.
Hence $1=u_1u_2$ and $1+m=(1+m_1)(1+m_2)$. Since $(1+m_1)$ is irreducible in $T$, it is so in $R$ by \cite[Lemma~1.5(1)]{CMZ}. Consequently, $1+m$ has an irreducible divisor in $R$, and therefore $R$ is a Furstenberg domain.

 If $D$ and $T$ are strong Furstenberg domains, then it follows from Proposition~\ref{P1} and Theorem~\ref{Th1}(1) that $R$ is an AP-domain. Thus, by the first part and Proposition~\ref{P1},  $D$ is a strong Furstenberg domain.   
\smallskip

(4) Suppose that $R$ is a PDF-domain. Fix $0\neq m_0\in M$. Then, for every $p\in \mathcal{P}(D)$, we have $m_0 = p(p^{-1}m_0)$ with $p^{-1} m_0 \in  M$. By Lemma \ref{pr}, every $p\in \mathcal{P}(D)$ is a prime divisor of $m_0$ in $R$. Since $U(D)= D\cap U(R)$, $\mathcal{P}(D)$ must be  finite up to associates. 
\smallskip

(5) Suppose that $T$ is a PDF-domain and that $\mathcal{P}(D)$ is finite (up to associates). Let $x=d+m$ be a nonzero nonunit element of $R$, where $d\in D$ and $m\in M$. Since no element of $M$ is irreducible in $R$, every prime divisor of $x$ in $R$ is, up to multiplication of a $u\in U(R)$, has the form $d_0$ or $1+m_0$, where $d_0\in D$  and $m_0\in M$.
Since $\mathcal{P}(D)$ is finite  and $U(D)= D\cap U(R)$, Lemma~\ref{dr} implies that we have only a finite number of nonassociate prime divisors of $x$ of the first form (i.e., $d_0\in D$).   By a similar argument as in the proof of Theorem \ref{Th1} each prime divisor of $x$ of the form  $(1+m_0)$ is prime in $T$. The fact that, for every $m_0,m_0'\in M$,  the elements $(1+m_0)$ and $(1+m_0')$ are associate in $R$ if and only if they are so in $T$ together with the fact that $T$ is a PDF-domain implies that $x$ has only a finite number of nonassociate prime divisors of the second form. Therefore, $R$ is a PDF-domain. 
\smallskip

The part (6) (resp., (7)) follows Follows from the parts (2) and (4) (resp., (3) and (5)) since a TPDF-domain is a strong Furstenberg PDF-domain. 
\end{proof}
\begin{rk} \label{r48} Under the assumptions of Theorem \ref{tN}, 
if T is a quasilocal domain, then each element of the form  $1+m$, for every $m\in M$, is a unit of $R$ (resp., of $T$). Then, in this case, we obtain the following stronger results.
\begin{enumerate}
\item  $R$ is a Furstenberg domain if and only if $D$ is a Furstenberg domain. \smallskip 

\item  $R$ is a PDF-domain if and only if $\mathcal{P}(D)$ is finite up to associates.
\smallskip 
\end{enumerate}
If in addition $T$ is an AP-domain, then the following statements hold.
\begin{enumerate}
\setcounter{enumi}{2}
\item $R$ is a strong Furstenberg domain if and only if $D$ is a strong Furstenberg domain.
		\smallskip
		 
\item $R$ is a TPDF-domain if and only if $D$ is a TPDF-domain with $\mathcal{P}(D)$ is  (nonempty) finite up to associates.  
\end{enumerate}
\end{rk}

\begin{Co}
 Let $D$ be an integral domain with quotient field $K$, and let $L$ be a field extension of $K$. Let $R = D + X L[X]$ and $R' = D + X L \ldb X \rdb$. Assume that $D$ is not a field. Then $R$ (resp.,  $R'$) is an TPDF-domain  if and only if  $D$ is a TPDF-domain with $\mathcal{P}(D)$ is finite  up to associates.
\end{Co}
\begin{proof} The direct implication follows from the part (6) of Theorem ~\ref{tN}. The converse follows  from Theorem~\ref{tN}(7) because $T=L[X]$ (resp., $T'=L \ldb X \rdb$) is a  TPDF-domain (since it is a UFD).
\end{proof}

\begin{Co} \label{fd}
		Let $T$ be an integral domain of the form $K+M$, where  $K$ is a subfield of $T$ and $M$ is a nonzero maximal ideal of $T$.  Let $D$ be a subfield  of $K$ and $R = D + M$.  Then the following statements hold.
	\begin{enumerate}
		\item If $\mathcal{A}(R) \cap M \neq \emptyset$, then $R$ is a TPDF-domain if and only if $T$ is a TPDF-domain, each element of $\mathcal{A}(R)\cap M$ is prime in $R$, and $|K^*/D^*| < \infty$.  
		\smallskip
		
		\item If  $\mathcal{A}(R) \cap M = \emptyset$, then  $R$ is a TPDF-domain if and only if $T$ is a TPDF-domain.
	\end{enumerate}
\end{Co}

\begin{proof} Follows from Theorem \ref{Th} and  \cite[Theorem~5.3]{GZ22}.
\end{proof}

Next, we investigate how the TPDF property behaves under localization. A saturated multiplicative subset $S$ of an integral domain $D$ is called \emph{splitting} if each $x \in D$ can be written as $x = as$ for some $a \in D$ and $s \in S$ such that $aD \cap tD = a tD$ for all $t \in S$. Note that the extension $D \subseteq D_S$ is inert when $S$ is a splitting multiplicative set, see \cite[Proposition~1.5]{AAZ92}.

\begin{Pro} \label{proAP}
	Let $D$ be an integral domain, and let $S$ be a splitting multiplicative subset of $D$ generated by primes. Then $D$ is an AP-domain (resp., a PDF-domain)  if and only if $D_S$ is an AP-domain (resp., a PDF-domain). 
\end{Pro}
\begin{proof}
 Assume that $D$ is an AP-domain. Let $x$ be an irreducible element of $D_S$. We may suppose that $x\in D$ and $xD \cap sD = xsD$ for all $s\in S$. Then it follows from \cite[Corollary~1.4(c)]{AAZ92} that $x$ is irreducible in $D$. Since $D$ is an AP-domain, $x$ is prime in $D$. Thus, by  \cite[Corollary~1.4(c)]{AAZ92} again, $x$ is prime in $D_S$. Hence  $D_S$ is an AP-domain. Conversely, suppose that $D_S$ is an AP-domain. Let $x$ be an irreducible element of $D$.  Note that every element $r \in D$ can be written as $r = as$ for some $a \in D$ and $s \in S$ such that no prime $p\in S$ divides $a$ in $D$. If $x\in S$, then $x$ is an associate of a prime element $p\in S$, and so $x$ is prime in $D$. Otherwise, $x\notin S$, and hence $x$ is irreducible in $D_S$ by \cite[Lemma~1.1]{AAZ92}. Since $D_S$ is an AP-domain, $x$ is prime in $D_S$. Then we may suppose that $xD \cap sD = xsD$ for all $s\in S$. Hence,  by  \cite[Corollary~1.4(c)]{AAZ92}, $x$ is prime in $D$. Therefore  $D$ is an AP-domain.
 \smallskip
 
The PDF case is similar to the IDF case; see \cite[Theorem~3.1]{AAZ92} and \cite[Theorem~2.4(a)]{AAZ92}.  
\end{proof}
\begin{Co}
	Let $D$ be an integral domain, and let $S$ be a splitting multiplicative subset of $D$ generated by primes. Then $D$ is a strong Furstenberg (resp., TPDF-domain)  if and only if $D_S$ is a strong Furstenberg (resp., TPDF-domain).	 
\end{Co}
\begin{proof}
By the proof of \cite[Proposition~6.4]{GZ22},  $D$ is a Furstenberg domain if and only if $D_S$ is a Furstenberg domain. On the other hand, a strong Furstenberg domain is a Furstenberg AP-domain  (see Propositions \ref{P1}). Then it follows from Proposition \ref{proAP} that $D$ is a strong Furstenberg domain if and only if $D_S$ is a strong Furstenberg domain. The TPDF-domain case follows from Proposition  \ref{proAP} and \cite[Proposition~6.4]{GZ22} since  a TPDF-domain is an AP TIDF-domain by Corollary \ref{pr1}. 
\end{proof}

We conclude this paper with the following proposition, which provides a class of examples of TPDF-domains and shows that such domains can be constructed with a prescribed number of primes, even when unique factorization fails.

\begin{Pro} \label{Pf}
    For every $n \in \N $, there exists a TPDF-domain that is not a UFD and has exactly $n$ prime elements up to associates.
\end{Pro}

\begin{proof}
Consider a UFD $D$ that is not a field with exactly $n$ nonassociate prime elements $\{p_1,\ldots,p_n\}$. 
For instance, let $p_1, p_2, \dots, p_n$ be distinct prime numbers in $\mathbb{Z}$ and define the multiplicative set
\[
S = \left\{ m \in \mathbb{Z}\setminus\{0\} : \gcd\!\left(m,\prod_{i=1}^{n} p_i\right)=1 \right\}.
\]
Now consider the localization $D=\mathbb{Z}_S$. Then $D$ is a subring of $\mathbb{Q}$ and hence an integral domain. Moreover, the group of units of $D$ is $U(D)=S$. It follows from \cite[Corollary~2.2]{AAZ92} that $D$ is a UFD. Furthermore, for each $i=1,\dots,n$, the element $p_i$ is not inverted in the localization and therefore remains a prime element of $D$. On the other hand, if $q$ is a prime number different from $p_1,\dots,p_n$, then $q\in S$ and hence becomes a unit in $D$. Consequently, the only prime elements of $D$, up to associates, are $p_1,\dots,p_n$.

Let $K$ be the quotient field of $D$ and set $R := D + XK[[X]]$. Since $T := K[[X]]$ is a UFD and a quasilocal domain with maximal ideal $M:=(X)$, it follows from Remark~\ref{r48}(4) that $R$ is a TPDF-domain. However, $D$ is not a field, and hence $R$ is not atomic  by \cite[Proposition~1.2]{AAZ90}. To conclude our proof, we need to show that $R$ has exactly $n$ nonassociate  irreducible (prime) elements, namely $p_1, \dots, p_n$. By Lemma~\ref{dr}(1), each $p_i\in \mathcal{A}(R)$, for $i=1,\ldots,n$.   On the other hand,  by  \cite[Lemma~4.17]{AG20},  $U(R) = U(D) + X K\ldb  X \rdb$. Then it is easy to see that every nonunit of $M$ is divisible by every $p_i$ in $R$, and  in particular $\mathcal{A}(R) \cap M = \emptyset$.   Therefore, every irreducible element of $R$ is of the form $r=d+m$ where $0 \neq d\in D\setminus U(D)$ and $m\in M$. Since $r=d(1+d^{-1}m)$ with $1+d^{-1}m\in U(R)$, $r$ is irreducible in $R$ if and only if $d$ is irreducible in $D$. Consequently,
$$\mathcal{A}(R) = \left\{ up_i : u\in U(R) \textit{ and }i=1,\ldots,n \right\}.
$$
Since $R$ is an AP-domain, the result follows. 
\end{proof}


\begin{thebibliography}{99}

\bibitem{AAZ90} D.~D. Anderson, D.~F. Anderson, and M.~Zafrullah, \emph{Factorization in integral domains}, J. Pure Appl. Algebra \textbf{69} (1990) 1--19.

\bibitem{AAZ92} D. D. Anderson, D. F. Anderson, and M. Zafrullah, \emph{Factorization in integral domains II}, J. Algebra \textbf{152} (1992) 78--93.


\bibitem{AZ07} D. D. Anderson and M. Zafrullah, \textit{The Schreier property and Gauss Lemma}, Bollettino U. M. I.,  Series 8,  \textbf{10-B} (2007) 43--62.

\bibitem{AG20} D.F. Anderson and F. Gotti, \emph{Bounded and finite factorization domains}, in The International Conference on Mathematics and Statistics, Springer, Singapore, 2020, pp. 7-57.


\bibitem{AS75} J. T. Arnold and P. B. Sheldon, \emph{Integral domains that satisfy Gauss’s lemma}, Michigan Math. J. \textbf{22} (1975) 39--51.

\bibitem{BEGraz} M. Benelmekki and  S. El Baghdadi, When is a group algebra  antimatter?, in  \emph{ Algebraic, Number Theoretic, and Topological Aspects of Ring Theory}, eds. J.L. Chabert et al.,  Springer, 2023, pp. 87--98. 






\bibitem{BDKN} A. Bodin, P. Debès, J. König, and S. Najib, \emph{The Hilbert-Schinzel specialization property}, J. Reine Angew. Math. \textbf{785} (2022) 55--79.

\bibitem{Cl17} P.L. Clark, \emph{The Euclidean criterion for irreducibles}, Am. Math. Mon. \textbf{124} (2017) 198--216.


\bibitem{C68} P. M. Cohn, \emph{Bezout rings and their subrings}, Proc. Cambridge Philos. Soc. \textbf{64} (1968) 251--264.

\bibitem{CDM99}  J. Coykendall, D. E. Dobbs, and B. Mullins, \emph{On integral domains with no atoms},
Comm. Alg. \textbf{27} (12) (1999) 5813--5831.


\bibitem{CMZ} D. Costa, J. L. Mott, and M. Zafrullah, \emph{Overrings and dimensions of general D + M constructions}, J. Natur. Sci.Math. \textbf{26} (1986) 7--14.



\bibitem{G84} R. Gilmer, \emph{Commutative Semigroup Rings}, Univ. of Chicago Press, Chicago, 1984.

\bibitem{GP74} R. Gilmer and T. Parker, \emph{Divisibility properties in semigroup rings}, Michigan Math. J., \textbf{21} (1974) 65--86

\bibitem{GZ22} F. Gotti and M. Zafrullah, \emph{Integral domains and the IDF property}, J. Algebra \textbf{614 }(2023) 564--591. 

	\bibitem{GW75} A. Grams and H. Warner, \emph{Irreducible divisors in domains of finite character}, Duke Math. J. \textbf{42} (1975) 271--284.
	

\bibitem{MO} P. Malcolmson and F. Okoh, \emph{Expansions of Prime Ideals}, Rocky Mt. J. Math. \textbf{35} (2005) 1689-1706.
 

\bibitem{MR78} S. McAdam and D. E. Rush, \emph{Schreier rings}, Bull. London Math. Soc. no. \textbf{10}
(1978) 77--80.



\bibitem{Sh71} P. B. Sheldon, \textit{How changing $D[[X]]$ changes its quotient field}, Trans. Amer. Math. Soc. \textbf{159}(1971) 223--244.

\bibitem{Tang} H. T. Tang, \emph{Gauss’ Lemma}, Proc. Amer. Math. Soc. \textbf{35} (1972) 372--376.

\end{thebibliography}
\end{document}